\documentclass[12pt,reqno]{amsart}
\usepackage{amssymb, amsmath}
\newtheorem{thm}{Theorem}

\newtheorem{defn}{Definition}

\numberwithin{defn}{section}
\numberwithin{thm}{section}
\numberwithin{Lemma}{section}
\numberwithin{Corollary}{section}
\numberwithin{Example}{section}
\numberwithin{subsection}{section}
\numberwithin{Remark}{section}
\numberwithin{equation}{section}
\numberwithin{ppn}{section}
\begin{document}
\title[ A New  Efficient Optimal Eighth-Order  Iterative Method. . .]
{A New Efficient Optimal Eighth-Order Iterative Method for Solving Nonlinear Equations}
\author{J. P. Jaiswal and Neha Choubey }
\date{}
\maketitle


\textbf{Abstract.} 
We established a new eighth-order iterative method, consisting of three steps, for solving nonlinear equations. Per iteration the method requires four evaluations (three function evaluations and one evaluation of the first derivative). Convergence analysis shows that this method is eighth-order convergent which is also substantiated through the numerical works. 
Computational results ascertain that our method is efficient and demonstrate almost better performance as compared to the other well known eighth-order methods.\\ 



\section{Introduction}
This paper concerns the numerical solution of non-linear equations of the general form $f(x)=0$. Such equations appear in real world situations frequently while there is no closed form solution for them. That is why the numerical solution of these types of equations draw much attention to itself day by day. One of the common problems encountered in science and engineering problems is that given a single variable function $f(x)$, find the values of $x$ for which $f(x)=0$. 
 The root of such nonlinear equations may be real or complex. 
 There are two general types of methods available to find the roots of algebraic and transcendental equations. First, direct methods, which are not always applicable to find the roots, and second, iterative methods based on the concept of successive approximations. In this case, the general procedure is to start with one or more initial approximation(s) to the root and attain a sequence of iterates, which in the limit converges to the true solution. Multipoint iterative methods for solving nonlinear equations are of great practical importance since they overcome theoretical limits of one-point methods concerning the convergence order and computational efficiency. Here, we focus on the simple root of nonlinear scalar equations by iterative method. Let the function $f: D\subseteq \Re \longrightarrow \Re$ be a sufficiently differentiable function and $\alpha \in D$ be a simple root of $f(x)=0$. The famous Newton's method \cite{Iliev} of order two which can be defined as $x_{n+1}=x_{n}-f(x_n)/f'(x_n)$ is one of the oldest and the most applicable method in the literature. Traub \cite{Traub} proposed the concept of efficiency index as a measure for comparing methods. This index is prescribed by $p^{1/n}$, where $p$ is the order of convergence and $n$ is the whole number of evaluations per iteration. Kung and Traub \cite{Kung} then presented a hypothesis on the optimality of roots by giving $2^{n-1}$ as the optimal order. This means that the Newton iteration by two evaluations per iterations is optimal with 1.414 as the efficiency index. By taking into account the optimality concept many authors have tried to build iterative methods of optimal higher order of convergence. Anyway, these schemes are divided into two main categories. First the derivative free and second, the methods in which (first, second, ...) derivative evaluation is used per cycle. In this paper we present a new eighth-order iterative method to find a simple root $\alpha$ of the nonlinear equation $f(x)=0$. We will compare our new method with well known existing eighth-order methods, namely proposed in \cite{Bi1}, \cite{Sharma}, \cite{Thukral}, \cite{Sargolzaei}, \cite{Kim}, \cite{Soleymani1}, \cite{Soleymani2}, \cite{Cordero2},  \cite{Wang}, .

This paper is organized as follows: In section 2, we describe the eighth-order method and prove that the method obtained preserves their convergence order. This new method agree with the Kung and Traub conjecture for $n=4$. In section 3, we will briefly state the well known established methods in order to compare the effectiveness of the new method. Finally, in section 4, 
the new method is compared in the performance with some well known eighth-order methods. Numerical results indicate that the our proposed method give better performance. Consequently, we have found that the new eighth-order method is consistent, stable and convergent.


\section{Development of the method and analysis of convergence}
In this section, we will define a new eighth-order method. 
In order to establish the order of convergence of this new method, we state following  definitions:\\\\

\begin{defn}
Let f(x) be a real function with a simple root $\alpha$ and let ${x_n}$ be a sequence of real numbers that converge towards $\alpha$. The order of convergence m is given by
\begin{equation}\label{eqn:21}
\lim_{n\rightarrow\infty}\frac{x_{n+1}-\alpha}{(x_n-\alpha)^m}=\zeta\neq0,  
\end{equation}     
\noindent
where $\zeta$ is the asymptotic error constant and $m \in R^+$.\\
\end{defn}


\begin{defn}
Let $\beta$ be the number of function evaluations of the new method. The efficiency of the new method is measured by the concept  of efficiency index \cite{Gautschi,Traub1} and defined as
\begin{equation}\label{eqn:23}
\mu^{1/\beta},
\end{equation}
where $\mu$ is the order of the method.\\
\end{defn}

\subsection{ New Eighth-order Method}
In this section, we construct our eight-order method by considering three step cycles. Kung \cite{Richard} developed a one-parameter family of fourth-order methods, which is written as
\begin{equation}\label{eqn:24}
y_n=x_n-\frac{f(x_n)}{f'(x_n)},
\end{equation}
\begin{equation}\label{eqn:25}
z_n=y_n-\frac{f(x_n)+\beta f(y_n)}{f(x_n)+(\beta-2)f(y_n)}\left(\frac{f(y_n)}{f'(x_n)}\right),
\end{equation} 
where $\beta$ is a constant. In particular, the special method for $\beta=-1/2$ is as follows:
\begin{equation}\label{eqn:26}
y_n=x_n-\frac{f(x_n)}{f'(x_n)},
\end{equation}
\begin{equation}\label{eqn:27}
z_n=y_n-\frac{2f(x_n)-f(y_n)}{2f(x_n)-5f(y_n)}\left(\frac{f(y_n)}{f'(x_n)}\right).
\end{equation} 
To achieve an $8^{th}$ order method, we use the first and second steps of the three step cycle from the above method and in the third step we apply the Newtons iteration 
\begin{equation}\label{eqn:28}
x_{n+1}=z_n-\frac{f(z_n)}{f'(z_n)}.
\end{equation}
As we can see, this method consists of three evaluation of the function and two evaluation of the first derivative per iteration, which has efficiency index 1.5157. To derive a scheme with a higher efficiency index, we approximate $f'(z_n)$ using a Hermite interpolation polynomial which is equal to the function $f(x)$ in the domain $D$ ($D$ is the interval in which $f$ has a simple root).

To approximate $f'(z_n)$, we construct a Hermite interpolation polynomial, $H(x)$, that meets the interpolation conditions
\begin{equation*}
H(x_n)=f(x_n),\  H(y_n)=f(y_n), \  H(z_n)=f(z_n)\  and \ H'(x_n)=f'(x_n). 
\end{equation*}
The $H(x)$ could be defined as follows:
\begin{equation}\label{eqn:29}
H(x)=\omega_0(x)f(x_n)+\omega_1(x)f(y_n)+\omega_2(x)f(z_n)+\overline{\omega}_0(x)f'(x_n), 
\end{equation}
where the interpolation basis functions $\omega_0(x)$, $\omega_1(x)$, $\omega_2(x)$ and $\overline{\omega}_0(x)$ are cubic polynomial function that satisfy in the following conditions 
\begin{eqnarray}\label{eqn:210}
(i)\ \omega_0(x_n)=1, \omega_0(y_n)=0, \omega_0(z_n)=0, \omega_0'(x_n)=0, \nonumber \\
(ii)\ \omega_1(x_n)=0, \omega_1(y_n)=1, \omega_1(z_n)=0, \omega_1'(x_n)=0,  \nonumber \\
(iii)\ \omega_2(x_n)=0, \omega_2(y_n)=0, \omega_2(z_n)=1, \omega_2'(x_n)=0,   \nonumber \\
(iv)\ \overline{\omega}_0(x_n)=1, \overline{\omega}_0(y_n)=0, \overline{\omega}_0(z_n)=0, \overline{\omega}_0'(x_n)=1. 
\end{eqnarray}
From $(i)$, we have $\omega_0(x)=A(x-y_n)(x-z_n)(x-B)$, where $A$, $B$ are constants. The conditions $\omega_0(x_n)=1$ and
$\omega_0'(x_n)=0$ imply
\begin{equation*}
A=-\frac{2x_n-z_n-y_n}{(x_n-y_n)^2(x_n-z_n)^2},\ \  B=x_n+\frac{(x_n-y_n)(x_n-z_n)}{2x_n-z_n-y_n}. 
\end{equation*}
Thus 
\begin{equation}\label{eqn:211}
\omega_0(x)=\frac{(x-x_n)(x-z_n)}{(x_n-y_n)(x_n-z_n)}\left[1-\frac{(x-x_n)(2x_n-z_n-y_n)}{(x_n-y_n)(x_n-z_n)}\right]. 
\end{equation}
From $(ii)$, we have $\omega_1(x)=C(x-x_n)^2(x-z_n)$ where $C$ is a constant. The condition $\omega_1(y_n)=1$
implies 
\begin{equation*}
C=\frac{1}{(y_n-x_n)^2(y_n-z_n)}.
\end{equation*}
Thus 
\begin{equation}\label{eqn:212}
\omega_1(x)=\frac{(x-x_n)^2(x-z_n)}{(y_n-x_n)^2(y_n-z_n)}. 
\end{equation}
Similarly,
\begin{equation}\label{eqn:213}
\omega_2(x)=\frac{(x-x_n)^2(x-y_n)}{(z_n-x_n)^2(z_n-y_n)}. 
\end{equation}
From $(iv)$, we have $\overline{\omega}(x)=D(x-x_n)(x-y_n)(x-z_n)$ where $D$ a constant. The conditions $\overline{\omega}_0'(x_n)=1$  gives
\begin{equation*}
D=\frac{1}{(x_n-y_n)(x_n-z_n)}.
\end{equation*}
Hence
\begin{equation}\label{eqn:214}
\overline{\omega}_0(x)=\frac{(x-x_n)(x-y_n)(x-z_n)}{(x_n-y_n)(x_n-z_n)}. 
\end{equation}
With the specific expressions of the interpolation basis functions $\omega_0(x)$,  $\omega_1(x)$ $\omega_2(x)$, $\overline{\omega}_0(x)$, we have
\begin{equation}\label{eqn:215}
\begin{split}
H(x) &=\frac{(x-y_n)(x-z_n)}{(x_n-y_n)(x_n-z_n)}\left[1-\frac{(x-x_n)(2x_n-y_n-z_n)}{(x_n-y_n)(x_n-z_n)}\right]f(x_n)\\
&+\frac{(x-x_n)^2(x-z_n)}{(y_n-x_n)(y_n-z_n)}f(y_n)+\frac{(x-x_n)^2(x-y_n)}{(z_n-x_n)^2(z_n-y_n)}f(z_n) \\
&+\frac{(x-x_n)(x-y_n)(x-z_n)}{(x_n-y_n)(x_n-z_n)}f'(x_n),
\end{split}
\end{equation}
Differentiating the above equation then putting $x=z_n$ and Simplifying we can get
\begin{equation}\label{eqn:216a}
\begin{split}
H'(z_n)&=2\frac{f(x_n)-f(z_n)}{x_n-z_n}+\frac{f(y_n)-f(z_n)}{y_n-z_n}-\frac{f(x_n)-f(y_n)}{x_n-y_n}\\
       &+\frac{(x_n-z_n)}{(y_n-x_n)}\left\{\frac{f(y_n)-f(x_n)}{(y_n-x_n)}\right\}-\frac{(x_n-z_n)}{(y_n-x_n)}f'(x_n)-f'(x_n), 
\end{split}
\end{equation}
and subsequently, we can find
\begin{equation}\label{eqn:217a}
H'(z_n)=2f[x_n,z_n]+f[y_n,z_n]-f[x_n,y_n]+(x_n-z_n)f[y_n,x_n,x_n]-f'(x),
\end{equation}

where $f[x_n,z_n]$ (similarly $f[y_n,z_n]$ and $f[x_n,y_n]$) and $f[y_n,x_n,x_n]$ are defined below by taking into consideration the divided differences:
\begin{equation*}
f[x_n,z_n]=\frac{f(x_n)-f(z_n)}{x_n-z_n},
\end{equation*}  
\begin{equation*}
f[y_n,x_n,x_n]=\frac{f[y_n,x_n]-f'(x_n)}{y_n-x_n}.
\end{equation*}  
Now replacing $f'(z_n)$ in the equation $(\ref{eqn:28})$ by $H'(z_n)$ from the equation $(\ref{eqn:217a})$, we find that
\begin{equation}\label{eqn:218a}
x_{n+1}=2f[x_n,z_n]+f[y_n,z_n]-f[x_n,y_n]+(x_n-z_n)f[y_n,x_n,x_n]-f'(x),
\end{equation}
Now this method consists of three evaluation of the functions and one evaluation of the first derivative per iteration, now it has improved efficiency index 1.6817. Now we prove the convergence of this method by the following theorem:

\begin{thm}
 Let us consider $\alpha$ as the simple root of the nonlinear equation f(x)=0 in the domain D and assume that f(x) is sufficiently smooth in the neighborhood of the root. Then, the iterative scheme defined by $(\ref{eqn:26})$, $(\ref{eqn:27})$ and $(\ref{eqn:218a})$ is of local order eight and has the following error equation 
\begin{equation*}\label{eqn:215}
e_{n+1}=(c_2^3c_3^2-c_2^3c_3c_4)e_n^8+O(e_n^9),
\end{equation*}
where $e_n=x_n-\alpha$ and $c_h=\frac{f^{(h)}(d)}{h!}$, h=1,2,3.... \\
\end{thm}

\begin{proof}
 We provide the Taylor series expansion of each term involved in $(\ref{eqn:26})$, $(\ref{eqn:27})$ and $(\ref{eqn:218a})$. By Taylor expansion around the simple root in the $n^{th}$ iteration, we have\\
\begin{equation}\label{eqn:216}
\begin{split}
f(x_n) &=f'(\alpha)[e_n+c_2e_n^2+c_3e_n^3+c_4e_n^4+c_5^5e_n^5\\
            &+c_6e_n^6+c_7e_n^7+c_8e_n^8+c_9e_n^9+c_{10}e_n^{10}\\
            &+c_{11}e_n^{11}+c_{12}e_n^{12}+O(e_n^{13})]
\end{split}
\end{equation}
and, we have
\begin{equation}\label{eqn:217}
\begin{split}
f'(x_n) &=f'(\alpha)[1+2c_2e_n+3c_3e_n^2+4c_4e_n^3+5c_5^5e_n^4\\
            &+6c_6e_n^5+7c_7e_n^6+8c_8e_n^7+9c_9e_n^8+10c_{10}e_n^{9}\\
            &+11c_{11}e_n^{10}+12c_{12}e_n^{11}+O(e_n^{12})].
\end{split}
\end{equation}
Further more it can be easily find 
\begin{equation}\label{eqn:218}
\frac{f(x_n}{f'(x_n)}=e_n-c_2e_n^2+2c_2e_n^3-2c_3e_n^3+........+O(e_n^{12}).
\end{equation}
By considering this relation and equation $(\ref{eqn:26})$, we obtain
\begin{equation}\label{eqn:219}
y_n=\alpha+c_2e_n^2+2(c_3-c_2^2)e_n^3+......+O(e_n^{12}).
\end{equation}
At this time, we should expand $f(y_n)$ around the root by taking into consideration $(\ref{eqn:219})$. Accordingly, we have
\begin{equation}\label{eqn:220}
f(y_n)=f'(\alpha)[c_2e_n^2+2(-c_2^2+c_3)e_n^3+......+O(e^{12})].
\end{equation}
Using $(\ref{eqn:219})$, $(\ref{eqn:216})$, $(\ref{eqn:220})$ and $(\ref{eqn:218})$ in the equation $(\ref{eqn:27})$, we can obtain
\begin{equation}\label{eqn:221}
z_n=\alpha-c_2c_3e_2^4+(\frac{3}{2}c_2^4+2c_2^2c_3-2c_4c_2-2c_3^2)e_n^5+....+O(e_n^{12}).
\end{equation}
On the other hand, we have
\begin{equation}\label{eqn:222}
f(z_n)=f'(\alpha)[-c_2c_3e_n^4+(\frac{3}{2}c_2^2c_3-2c_4c_2-2c_3^2)e_n^5+......+O(e_n^{12})].
\end{equation} 
Now we expand the Taylor series of each existing divided differences in the denominator of equation $(\ref{eqn:218a})$. We obtain
\begin{equation}\label{eqn:223}
f[x_n,z_n]=f'(\alpha)[1+c_2e_n+c_3e_n^2+c_4e_n^3+.....+O(e_n^{12})],
\end{equation}
\begin{equation}\label{eqn:224}
f[y_n,z_n]=f'(\alpha)[1+c_2e_n^2+(2c_2c_3-2c_2^3)e_n^3+.....+O(e_n^{12})],
\end{equation}
\begin{equation}\label{eqn:225}
f[x_n,y_n]=f'(\alpha)[1+c_2e_n+(c_2^2+c_3)e_n^2+(-2c_2^3+3c_3c_2+c_4)e_n^3+.....+O(e_n^{12})],
\end{equation}
\begin{equation}\label{eqn:226}
\begin{split}
f[y_n, x_n, x_n]&=f'(\alpha)[c_2+2c_3e_n+(3c_4+c_2c_3)e_n^2\\
                &+(-2c_2^2c_3+2c_4c_2+2c_3^2+4c_5)e_n^3+....+O(e_n^{12})].
\end{split}
\end{equation}
By consider the above mentioned relations $(\ref{eqn:223})$-$(\ref{eqn:226})$ in the equation $(\ref{eqn:218a})$, we can find
\begin{equation}\label{eqn:227}
\begin{split}
x_{n+1} &=(c_2^3c_3^2-c_2^2c_3c_4)e^8+[-3c_2^6c_3+\frac{3}{2}c_2^5c_4-4c_2^4c_3^2+8c_2^3c_3c_4\\
            &+4c_2^2c_3^3-2c_5c_2^2c_3-2c_2^2c_4^2-4c_2c_3^2c_4]e^9+...+O(e_n^{12}).\\
\end{split}
\end{equation}
\end{proof}
This implies that the order of convergence for this method is eight i.e. optimal order of convergence and its efficiency index is $8^{1/4}\approx1.6817$, which is more that 1.4142 of Newton's method, and 1.5650 of three-step methods \cite{Chun,Cordero1}, and is equal to 1.6817 of \cite{Bi1}, \cite{Sharma}, \cite{Thukral}, \cite{Sargolzaei}, \cite{Kim}, \cite{Soleymani1}, \cite{Soleymani2}, \cite{Cordero2},  \cite{Wang} but numerical performance is  better to almost all the eighth-order methods ( later shown in the Tables 2-8).

\section{Well established eighth-order Methods}
First we are giving here some well established eighth-order methods:

\subsection{ Bi et al. Methods, [ \cite{Bi1}, 2009]:\\}
\textit{Method I}
\begin{equation*}\label{eqn:31}
y_n=x_n-\frac{f(x_n)}{f'(x_n)},
\end{equation*}
\begin{equation*}\label{eqn:32}
z_n=y_n-\left[\frac{2f(x_n)-f(y_n)}{2f(x_n)-5f(y_n)}\right]\left(\frac{f(y_n)}{f'(x_n)}\right),
\end{equation*}
\begin{equation}\label{eqn:33}
x_{n+1}=z_n-\left[\frac{f(x_n)+(\gamma+2)f(z_n)}{f(x_n)+\gamma f(z_n)}\right]\left(\frac{f(z_n)}{f[z_n,y_n]+f[z_n,x_n,x_n](z_n-y_n)}\right).\\
\end{equation}

\textit{Method II}
\begin{equation*}\label{eqn:34}
z_n=y_n-\left[1+\frac{2f(y_n)}{2f(x_n)}+5\left(\frac{f(y_n)}{f(x_n)}\right)^2+\left(\frac{f(y_n)}{f(x_n)}\right)^3\right]\left(\frac{f(y_n)}{f'(x_n)}\right),
\end{equation*}
\begin{equation}\label{eqn:35}
x_{n+1}=z_n-\left[\frac{f(x_n)+(\gamma+2)f(z_n)}{f(x_n)+\gamma f(z_n)}\right]\left(\frac{f(z_n)}{f[z_n,y_n]+f[z_n,x_n,x_n](z_n-y_n)}\right).\\
\end{equation}

\textit{Method II}
\begin{equation*}\label{eqn:36}
z_n=y_n-\left[1-\frac{2f(y_n)}{2f(x_n)}-\left(\frac{f(y_n)}{f(x_n)}\right)^2+\left(\frac{f(y_n)}{f(x_n)}\right)^3\right]^{-1}\left(\frac{f(y_n)}{f'(x_n)}\right),
\end{equation*}
\begin{equation}\label{eqn:37}
x_{n+1}=z_n-\left[\frac{f(x_n)+(\gamma+2)f(z_n)}{f(x_n)+\gamma f(z_n)}\right]\left(\frac{f(z_n)}{f[z_n,y_n]+f[z_n,x_n,x_n](z_n-y_n)}\right).\\
\end{equation}

\textit{Method IV}
\begin{equation*}\label{eqn:38}
z_n=y_n-\left[\frac{2f(x_n)-3f(y_n)}{f(x_n)}\right]^{-2/3}\left(\frac{f(y_n)}{f'(x_n)}\right),
\end{equation*}
\begin{equation}\label{eqn:39}
x_{n+1}=z_n-\left[\frac{f(x_n)+(\gamma+2)f(z_n)}{f(x_n)+\gamma f(z_n)}\right]\left(\frac{f(z_n)}{f[z_n,y_n]+f[z_n,x_n,x_n](z_n-y_n)}\right),
\end{equation}
where $\gamma \in R$ and denominator is not equal to zero.\\

\subsection{ Sharma et al. Methods, [ \cite{Sharma}, 2010]:\\}

\textit{Method I}
\begin{equation*}\label{eqn:310}
z_n=y_n-\left[\frac{f(x_n)}{f(x_n)-2f(y_n)}\right]\left(\frac{f(y_n)}{f'(x_n)}\right),
\end{equation*}
\begin{equation}\label{eqn:311}
x_{n+1}=z_n-\left[1+\frac{f(z_n)}{f(x_n)}+\gamma \left(\frac{f(z_n)}{f(x_n)}\right)^2\right]\left(\frac{f[x_n,y_n]f(z_n)}{f[y_n,z_n]f[x_n,z_n]}\right).\\
\end{equation}

\textit{Method II}
\begin{equation*}\label{eqn:312}
z_n=y_n-\left[\frac{f(x_n)}{f(x_n)-2f(y_n)}\right]\left(\frac{f(y_n)}{f'(x_n)}\right),
\end{equation*}
\begin{equation}\label{eqn:313}
x_{n+1}=z_n-\left[\frac{f(x_n)+(\gamma+1)f(z_n)}{f(x_n)+\gamma f(z_n)}\right]\left(\frac{f[x_n,y_n]f(z_n)}{f[y_n,z_n]f[x_n,z_n]}\right).\\
\end{equation} 

\textit{Method III}
\begin{equation*}\label{eqn:314}
z_n=y_n-\left[\frac{f(x_n)}{f(x_n)-2f(y_n)}\right]\left(\frac{f(y_n)}{f'(x_n)}\right),
\end{equation*}
\begin{equation}\label{eqn:315}
x_{n+1}=z_n-\left[1+\gamma\frac{f(z_n)}{f(x_n)}\right]^{(1/\gamma)}\left(\frac{f[x_n,y_n]f(z_n)}{f[y_n,z_n]f[x_n,z_n]}\right),
\end{equation}
where $\gamma \in R$ and denominator is not equal to zero.\\

\subsection{Thukral Method, [ \cite{Thukral}, 2010]:}
\begin{equation*}\label{eqn:317}
z_n=x_n-\frac{f(x_n)^2+f(y_n)^2}{f'(x_n)(f(x_n)-f(y_n))},
\end{equation*}
\begin{equation}\label{eqn:318}
x_{n+1}=z_n-\left[\left(\frac{1+\mu_n^2}{1-\mu_n}\right)^2-2(\mu_n)^2-6(\mu_n)^3+\frac{f(z_n)}{f(y_n)}+
4 \frac{f(z_n)}{f(x_n)} \right] \left(\frac{f(z_n)}{f'(x_n)}\right).
\end{equation}
where $\mu_n=(f(y_n)/f(x_n))$ and denominator is not equal to zero.\\

\subsection{ Wang et al., [ \cite{Wang}, 2010]:} 

\begin{equation*}\label{eqn:336}
z_n=y_n-\frac{f(y_n)}{2f[x_n,y_n]-f'(x_n)},
\end{equation*}
\begin{equation}\label{eqn:337}
x_{n+1}=z_n-\frac{f(z_n)}{2f[x_n,z_n]+f[y_n,z_n]-2f[x_n,y_n]+(y_n-z_n)f[y_n,x_n,x_n]}.
\end{equation}

\subsection{\bf  Sragolzaei et al. Method, [ \cite{Sargolzaei}, 2011]:}
\begin{equation*}\label{eqn:320}
z_n=y_n-\left(1+\frac{f(y_n)}{f(x_n)}\right)^2\left(\frac{f(y_n)}{f'(x_n)}\right),
\end{equation*}
\begin{equation}\label{eqn:321}
x_{n+1}=z_n-\frac{f(z_n)}{2f[z_n,x_n]+f[z_n,y_n]-2f[y_n,x_n]+(y_n-z_n)f[y_n,x_n,x_n]}.
\end{equation}

\subsection{Cordero et al. Method, [ \cite{Cordero2}, 2011]}
\begin{equation*}\label{eqn:332}
z_n=x_n-\frac{f(x_n)}{f'(x_n)}\left[\frac{f(x_n)-f(y_n)}{f(x_n)-2f(y_n)}\right],
\end{equation*}
\begin{equation*}\label{eqn:333}
u_n=z_n-\frac{f(z_n)}{f'(x_n)}\left[\frac{f(x_n)-f(y_n)}{f(x_n)-2f(y_n)}+\frac{1}{2}\frac{f(z_n)}{f(y_n)-2f(z_n)}\right]^2,
\end{equation*}
\begin{equation}\label{eqn:334}
x_{n+1}=u_n-3\frac{f(z_n)}{f'(x_n)}\left[\frac{u_n-z_n}{y_n-x_n}\right].
\end{equation}

\subsection{ Soleymani Methods, [(\cite{Soleymani2}, \cite{Soleymani1}); (2011, 2012)]:\\} 

\textit{Method I}
\begin{equation*}\label{eqn:330}
z_n=x_n-\left[1-\frac{3}{8}\frac{f'(y_n)^2-f'(x_n)^2}{f'(y_n)^2}\right]\frac{f(x_n)}{f'(x_n)},
\end{equation*} 
\begin{equation}\label{eqn:331}
x_{n+1}=z_n-\frac{f(z_n)}{f'(y_n)+2f[z_n,x_n,x_n](z_n-y_n)}.
\end{equation}
\textit{Method II}
\begin{equation*}\label{eqn:326}
z_n=x_n-\frac{2f(x_n)}{f'(x_n)+f'(y_n)},
\end{equation*}  
\begin{equation*}\label{eqn:327}
k_n=z_n-\frac{f(z_n)}{f'(y_n)},
\end{equation*} 
\begin{equation}\label{eqn:328}
x_{n+1}=k_n-\frac{f(k_n)}{\frac{f'(x_n)(3f'(y_n)-f'(x_n))}{f'(x_n)+f'(y_n)}}.\\
\end{equation} 
 
\subsection{\bf Kim Method, [ \cite{Kim}, 2012]:}
\begin{equation*}\label{eqn:323}
z_n=y_n-\left(\frac{1+\beta u_n+\lambda u_n^2}{1+(\beta-2)u_n+\mu u_n^2}\right)\frac{f(y_n)}{f'(x_n)},
\end{equation*}
\begin{equation}\label{eqn:324}
x_{n+1}=z_n-\left(\frac{1+a u_n+b v_n}{1+c u_n+d v_n}\right)\frac{f(z_n)}{f'(x_n)+f[y_n, x_n, z_n](z_n-x_n)},
\end{equation}
where $u_n=f(y_n)/f(x_n)$ , $v_n=f(z_n)/f(x_n)$ , $\beta,\  \lambda,\  \mu,\  a,\  b,\  c, $ and $d$ are constant parameters and related by $\beta=(\lambda-\mu-2/3)/2$, $a=-2$, $c=-3$, $d=b-3$.



\section{ Numerical Testing}
To demonstrate the performance of the new eighth-order method, we take seven particular non-linear equations. We will determine the consistency and stability of results by examining the convergence of the new iterative method. We will give estimates of the approximate solution produced be the eighth-order method.

Here we consider, the following test functions to illustrate the accuracy of new iterative method. The root of each nonlinear test function is also listed in from of each page up to fifteen decimal places, when such roots are non-integers. All the computations reported here we have done using Mathematica 8, 
Scientific computations in many branches of science and technology demand very high precision degree of numerical precision. The test non-linear functions are listed in Table-1.
\begin{table}[htb]
 \caption{ Test functions and their roots.}
  \begin{tabular}{ll} \hline
Non-linear function & \hspace{50pt}Roots \\ \hline 
$f_1(x)=\sin(x)-\frac{x}{100}$ & \hspace{50pt}0 \\ 
$f_2(x)=\frac{1}{3x^4}-x^3-\frac{1}{3x}+1$ &\hspace{50pt}1 \\ 
$f_3(x)=\exp(\sin(x))-1-\frac{x}{5}$ & \hspace{50pt}0 \\ 
$f_4(x)=x+\sin(\frac{x^2}{\pi})$ & \hspace{50pt}0\\ 
$f_5(x)=\sqrt{x^4+8}\sin(\frac{\pi}{x^2+2})+\frac{x^3}{x^4+1}-\sqrt{6}+\frac{8}{16}$ & \hspace{50pt}-2\\ 
$f_6(x)=\cos(x)-x$ & \hspace{50pt}0.739085133215160\\ 
$f_7(x)=\exp(x)+\cos(x)$ & \hspace{50pt}-1.7461395304080124\\ \hline
  \end{tabular}
  \label{tab:abbr}
\end{table}
The results of comparison for the test function are provided in the Table 2-8. It can be seen that the resulting method from our class are accurate and efficient in terms of number of accurate decimal places to find the roots after some iterations. 

\newpage
\begin{table}[htb]
 \caption{Errors Occurring in the estimates of the root of function $f_1$ by the method described with initial guess $x_0=0.7$.}
  \begin{tabular}{llll} \hline
Methods & $\left|f_1(x_1)\right|$&$\left|f_1(x_2)\right|$ & $\left|f_1(x_3)\right|$ \\\hline
$(\ref{eqn:218a})$ & 0.695e-5& 0.654e-60&0.336e-665 \\ 
$(\ref{eqn:33})$ with $\gamma=1$ &0.791e-3&0.363e-18&0.740e-95 \\ 
$(\ref{eqn:35})$ with $\gamma=1$ &0.205e-3&0.428e-21&0.169e-109\\ 
$(\ref{eqn:37})$ with $\gamma=1$ &0.628e-3&0.115e-18&0.233e-97 \\ 
$(\ref{eqn:39})$ with $\gamma=1$ &0.887e-2&0.177e-8&0.140e-28 \\
$(\ref{eqn:311})$ with $\gamma=1$ &0.689e-4&0.339e-48&0.139e-535 \\
$(\ref{eqn:313})$ with $\gamma=1$ &0.821e-4&0.233e-47&0.224e-526 \\
$(\ref{eqn:315})$ with $\gamma=1$ &0.754e-4&0.918e-48&0.790e-531 \\
$(\ref{eqn:318})$ &0.129e-2&0.174e-28&0.257e-261 \\    
$(\ref{eqn:321})$ &0.275e-4&0.239e-53&0.511e-593 \\  
$(\ref{eqn:324})$ with $(\lambda, \mu, b)=(0,0,4)$ &0.101e-3&0.414e-46&0.201e-512 \\ 
$(\ref{eqn:328})$ &0.148e-4&0.116e-65&0.512e-860 \\
$(\ref{eqn:331})$ &0.133e-2&0.601e-10&0.548e-32 \\ 
$(\ref{eqn:334})$ &0.830e-4&0.409e-39&0.702e-357 \\
$(\ref{eqn:337})$ &0.101e-4&0.414e-58&0.217e-645 \\   \hline
  \end{tabular}
  \label{tab:abbr}
\end{table}

\newpage                         
\begin{table}[htb]
 \caption{Errors Occurring in the estimates of the root of function $f_2$ by the method described with initial guess $x_0=1.2$.}
  \begin{tabular}{llll} \hline
Methods & $\left|f_2(x_1)\right|$&$\left|f_2(x_2)\right|$ & $\left|f_2(x_3)\right|$  \\\hline
$(\ref{eqn:218a})$ & 0.111e-3& 0.497e-31&0.823e-250 \\ 
$(\ref{eqn:33})$ with $\gamma=1$ &0.372e-3&0.709e-24&0.638e-161 \\ 
$(\ref{eqn:35})$ with $\gamma=1$ &dgt&dgt&dgt \\ 
$(\ref{eqn:37})$ with $\gamma=1$ &0.165e+0&.136e-3&0.442e-15 \\
$(\ref{eqn:39})$ with $\gamma=1$ &0.768e-2&0.116e-6&0.427e-21 \\
$(\ref{eqn:311})$ with $\gamma=1$ &0.317e-2&0.192e-18&0.376e-148 \\
$(\ref{eqn:313})$ with $\gamma=1$ &0.547e-2&0.141e-16&0.308e-133 \\
$(\ref{eqn:315})$ with $\gamma=1$ &0.420e-2&0.176e-17&0.186e-140 \\
$(\ref{eqn:318})$ &0.213e+1&0.536e-1&0.271e-7 \\   
$(\ref{eqn:321})$ &0.109e-1&0.431e-14&0.362e-113 \\    
$(\ref{eqn:324})$ with $(\lambda, \mu, b)=(0,0,4)$ &0.496e-4&0.101e-33&0.318e-271 \\ 
$(\ref{eqn:328})$ &0.667e-3&0.119e-24&0.128e-198 \\
$(\ref{eqn:331})$ &0.545e-4&0.419e-9&0.247e-19 \\ 
$(\ref{eqn:334})$ &0.992e-2&0.683e-15&0.467e-120 \\
$(\ref{eqn:337})$ &0.713e-3&0.409e-24&0.494e-194 \\ \hline 
  \end{tabular}
  \label{tab:abbr}
\end{table}

\newpage
\begin{table}[htb]
 \caption{Errors Occurring in the estimates of the root of function $f_3$ by the method described with initial guess $x_0=-0.55$.}
  \begin{tabular}{llll} \hline
Methods & $\left|f_3(x_1)\right|$&$\left|f_3(x_2)\right|$ & $\left|f_3(x_3)\right|$  \\\hline
$(\ref{eqn:218a})$ & 0.628e-2& 0.344e-20&0.168e-184 \\ 
$(\ref{eqn:33})$ with $\gamma=1$ &0.175e-1&0.682e-9&0.549e-146 \\ 
$(\ref{eqn:35})$ with $\gamma=1$ &0.266e+1&0.260e+2&0.527e+3 \\ 
$(\ref{eqn:37})$ with $\gamma=1$ &0.176e+0&0.736e-4&0.871e-17 \\ 
$(\ref{eqn:39})$ with $\gamma=1$ &0.486e-2&0.243e-5&0.614e-12 \\
$(\ref{eqn:311})$ with $\gamma=1$ &0.520e-2&.212e-18&0.173e-149 \\
$(\ref{eqn:313})$ with $\gamma=1$ &0.883e-2&0.141e-16&0.652e-135 \\
$(\ref{eqn:315})$ with $\gamma=1$ &0.684e-2&0.186e-17&0.610e-142 \\
$(\ref{eqn:318})$ &0.271e+1&0.626e+3&0.225e+5 \\
$(\ref{eqn:321})$ &0.107e-1&0.344e-15&0.582e-123 \\
$(\ref{eqn:324})$ with $(\lambda, \mu, b)=(0,0,4)$ &0.623e-2&0.337e-18&0.261e-148 \\  
$(\ref{eqn:328})$ &0.737e-2&0.568e-17&0.774e-138 \\ 
$(\ref{eqn:331})$ &0.424e-2&0.198e-5&0.431e-12 \\ 
$(\ref{eqn:334})$ &0.176e-1&0.284e-14&0.162e-116 \\  
$(\ref{eqn:337})$ &0.110e-2&0.141e-24&0.104e-199 \\  \hline
  \end{tabular}
  \label{tab:abbr}
\end{table}

\newpage                         
\begin{table}[htb]
 \caption{Errors Occurring in the estimates of the root of function $f_4$ by the method described with initial guess $x_0=0.1$.}
  \begin{tabular}{llll} \hline
Methods & $\left|f_4(x_1)\right|$&$\left|f_4(x_2)\right|$ & $\left|f_4(x_3)\right|$  \\\hline
$(\ref{eqn:218a})$ & 0.467e-14& 0.371e-147&0.370e-1478 \\ 
$(\ref{eqn:33})$ with $\gamma=1$ &0.340e-12&0.287e-115&0.629e-1043 \\ 
$(\ref{eqn:35})$ with $\gamma=1$ &0.199e-10&0.174e-97&0.513e-881 \\ 
$(\ref{eqn:37})$ with $\gamma=1$ &0.124e-11&0.128e-109&0.173e-991 \\ 
$(\ref{eqn:39})$ with $\gamma=1$ &0.274e-4&0.568e-15&0.508e-47 \\
$(\ref{eqn:311})$ with $\gamma=1$ &0.756e-11&0.106e-91&0.156e-738\\
$(\ref{eqn:313})$ with $\gamma=1$ &0.756e-11&0.106e-91&0.162e-738 \\
$(\ref{eqn:315})$ with $\gamma=1$ &0.756e-11&0.106e-91&0.159e-738 \\
$(\ref{eqn:318})$ &0.141e-9&0.322e-80&0.242e-645 \\   
$(\ref{eqn:321})$ &0.354e-10&0.130e-85&0.435e-689 \\  
$(\ref{eqn:324})$ with $(\lambda, \mu, b)=(0,0,4)$ &0.169e-11&0.145e-97&0.442e-786 \\ 
$(\ref{eqn:328})$ &0.101e-10&0.146e-90&0.283e-729 \\
$(\ref{eqn:331})$ &0.472e-3&0.999e-8&0.447e-17 \\ 
$(\ref{eqn:334})$ &0.681e-11&0.423e-92&0.950e-742 \\
$(\ref{eqn:337})$ &0.256e-11&0.618e-96&0.706e-773 \\ \hline
  \end{tabular}
  \label{tab:abbr}
\end{table}

\newpage
\begin{table}[htb]
 \caption{Errors Occurring in the estimates of the root of function $f_5$ by the method described with initial guess $x_0=-3$.}
  \begin{tabular}{llll} \hline
Methods & $\left|f_5(x_1)\right|$&$\left|f_5(x_2)\right|$ & $\left|f_5(x_3)\right|$  \\\hline
$(\ref{eqn:218a})$ & 0.283e-3& 0.555e-26&0.108e-207 \\ 
$(\ref{eqn:33})$ with $\gamma=1$ &0.756e-2&0.218e-7&0.129e-29 \\ 
$(\ref{eqn:35})$ with $\gamma=1$ &0.690e-2&0.143e-7&0.243e-30 \\ 
$(\ref{eqn:37})$ with $\gamma=1$ &0.745e-2&0.201e-7&0.107e-29 \\ 
$(\ref{eqn:39})$ with $\gamma=1$ &0.179e-1&0.159e-3&0.101e-7 \\
$(\ref{eqn:311})$ with $\gamma=1$ &0.333e-3&0.599e-25&0.659e-199 \\
$(\ref{eqn:313})$ with $\gamma=1$ &0.321e-3&0.447e-25&0.633e-200 \\
$(\ref{eqn:315})$ with $\gamma=1$ &0.328e-3&0.519e-25&0.209e-199 \\
$(\ref{eqn:318})$ &0.251e-3&0.409e-25&0.199e-199 \\   
$(\ref{eqn:321})$ &0.380e-3&0.157e-24&0.138e-195 \\ 
$(\ref{eqn:324})$ with $(\lambda, \mu, b)=(0,0,4)$ &0.432e-3&0.545e-25&0.344e-200 \\
$(\ref{eqn:328})$ &0.705e-3&0.362e-23&0.173e-185 \\
$(\ref{eqn:331})$ &0.622e-2&0.676e-5&0.779e-11 \\ 
$(\ref{eqn:334})$ &0.122e-2&0.507e-23&0.229e-185 \\
$(\ref{eqn:337})$ &0.307e-3&0.132e-25&0.160e-204 \\   \hline
  \end{tabular}
  \label{tab:abbr}
\end{table}


\newpage
\begin{table}[htb]
 \caption{Errors Occurring in the estimates of the root of function $f_6$ by the method described with initial guess $x_0=1.5$.}
  \begin{tabular}{llll} \hline
Methods & $\left|f_6(x_1)\right|$&$\left|f_6(x_2)\right|$ & $\left|f_6(x_3)\right|$ \\\hline
$(\ref{eqn:218a})$ & 0.696e-6& 0.176e-55&0.300e-452 \\ 
$(\ref{eqn:33})$ with $\gamma=1$ &0.682e-2&0.124e-10&0.132e-45 \\ 
$(\ref{eqn:35})$ with $\gamma=1$ &0.706e-2&0.141e-10&0.227e-45 \\ 
$(\ref{eqn:37})$ with $\gamma=1$ &0.686e-2&0.126e-10&0.143e-45 \\ 
$(\ref{eqn:39})$ with $\gamma=1$ &0.650e-1&0.383e-3&0.129e-7 \\
$(\ref{eqn:311})$ with $\gamma=1$ &0.415e-5&0.990e-48&0.103e-388 \\
$(\ref{eqn:313})$ with $\gamma=1$ &0.421e-5&0.111e-47&0.261e-388 \\
$(\ref{eqn:315})$ with $\gamma=1$ &0.418e-5&0.109e-47&0.164e-388 \\
$(\ref{eqn:318})$ &0.249e-4&0.203e-40&0.385e-329 \\    
$(\ref{eqn:321})$ &0.142e-5&0.222e-51&0.797e-418 \\
$(\ref{eqn:324})$ with $(\lambda, \mu, b)=(0,0,4)$ &0.228e-5&0.317e-50&0.454e-409 \\  
$(\ref{eqn:328})$ &0.234e-5&0.185e-50&0.289e-411 \\ 
$(\ref{eqn:331})$ &0.129e-1&0.311e-5&0.180e-12 \\ 
$(\ref{eqn:334})$ &0.893e-5&0.142e-45&0.589e-372 \\
$(\ref{eqn:337})$ &0.870e-6&0.363e-54&0.332e-441 \\   \hline
  \end{tabular}
  \label{tab:abbr}
\end{table}

\newpage
\begin{table}[htb]
 \caption{Errors Occurring in the estimates of the root of function $f_7$ by the method described with initial guess $x_0=-2.3$.}
  \begin{tabular}{llll} \hline
Methods & $\left|f_7(x_1)\right|$&$\left|f_7(x_2)\right|$ & $\left|f_7(x_3)\right|$ \\\hline
$(\ref{eqn:218a})$ & 0.563e-6& 0.167e-54&0.101e-442 \\ 
$(\ref{eqn:33})$ with $\gamma=1$ &0.105e-2&0.464e-14&0.179e-59 \\ 
$(\ref{eqn:35})$ with $\gamma=1$ &0.529e-2&0.301e-11&0.318e-48 \\ 
$(\ref{eqn:37})$ with $\gamma=1$ &0.490e-3&0.222e-15&0.946e-65 \\ 
$(\ref{eqn:39})$ with $\gamma=1$ &0.578e-2&0.554e-6&0.507e-14 \\
$(\ref{eqn:311})$ with $\gamma=1$ &0.412e-4&0.749e-39&0.893e-317 \\
$(\ref{eqn:313})$ with $\gamma=1$ &0.443e-4&0.135e-38&0.977e-315 \\
$(\ref{eqn:315})$ with $\gamma=1$ &0.428e-4&0.101e-38&0.960e-316 \\
$(\ref{eqn:318})$ &0.718e-3&0.478e-28&0.186e-229 \\ 
$(\ref{eqn:321})$ &0.788e-4&0.774e-37&0.671e-301 \\    
$(\ref{eqn:324})$ with $(\lambda, \mu, b)=(0,0,4)$ &0.231e-4&0.551e-42&0.580e-343 \\ 
$(\ref{eqn:328})$ &0.788e-6&0.357e-53&0.647e-432 \\ 
$(\ref{eqn:331})$ &0.384e-3&0.269e-8&0.132e-18 \\
$(\ref{eqn:334})$ &0.246e-4&0.910e-41&0.314e-332 \\
$(\ref{eqn:337})$ &0.592e-5&0.357e-46&0.630e-376 \\   \hline
  \end{tabular}
  \label{tab:abbr}
\end{table}
\newpage
\section{Conclusion}
In this work, we have developed a new eighth-order convergent method for solving non-linear equations. Convergence analysis shows that our new method is eighth-order convergent which is also supported by the numerical works. This iterative method require evaluation of three functions and one first derivative during each iterative step. Computational results demonstrate that the iterative method is efficient and exhibit better performance as compared with other well known eighth-order methods.

 {Jai Prakash Jaiswal\\
Department of Mathematics,\\
 Maulana Azad National Institute of Technology,\\ 
 Bhopal, M.P., India-462051}.\\
{Corresponding author:\\
 E-mail: {\sf jpbhu2007@gmail.com}.\\\\
 \textsc{Neha Choubey\\
Department of Applied Sciences and Humanities,\\
Oriental Institute of Science and Technology,\\ 
 Bhopal, M.P., India-462021}\\
 E-mail: {nehachby2@gmail.com}.


\begin{thebibliography}{10}


\bibitem{Cordero}
A. Cordero, S. L. Huero, E. Martinez and J. R. Torregrosa: Steffensen type methods for solving non-linear equations, Applied Mathematics and computation, vol 198, no. 2, pp. 527-533 (2007).
\bibitem{Cordero1}
A. Cordero, J. L. Huero, M. Martinez and J. R. Torregrosa: Efficient three step iterative methods with sixth order convergence for non linear equations, Numerical Algorithms 53, pp. 485-495 (2010).
\bibitem{Cordero2}
A. Cordero, J. R. Torregrosa and M. P. Vassileva: Three step iterative method with optimal eighth order convergence, Journal of Computational and Applied Mathematics, 235, pp. 3189-3194, (2011).
\bibitem{Iliev}
A. Iliev and N. Kyurtrchiev, Nontrivial method in Numerical Analysis selected topics in Numerical Analysis, LAP Lambert Academic Publishing (2010).
\bibitem{Yun}
B. I. Yun: A non-iterative method for solving non-linear equations, Applied Mathematics and Computation, vol. 198, no.2, pp. 691-699 (2008).
\bibitem{Chun}
C. Chun and Y. Ham: Some sixth-order variation of ostrowski's root finding methods, Applied Mathematical Computations, 193, pp. 389-394 (2007).
\bibitem{Soleymani2}
F. Soleymani: Regarding the accuracy of optimal eighth-order methods, Mathematical and Computer Modelling, 53, pp. 1351-1357  (2011).
\bibitem{Soleymani1}
F. Soleymani: A second derivative free eighth-order non-linear equation solver, Non-linear Stud. (19) No. 1, pp. 79-86 (2012).
\bibitem{Kung}
H. T. Kung and J. F. Traub: Optimal order of one-point and multipoint iteration, J.CAM 21, PP. 643-651 (1974).
\bibitem{Traub1} 
J. F. Traub: Iterative methods for solution of equations chelsea Publishing, New York, NY, USA (1997).
\bibitem{Traub}
J. F. Traub: Iterative methods for the solution of equation, Prentice-Hall, Englewood cliffs, New Jerset (1964).
\bibitem{Sharma}
J. R. Sharma and R. Sharma: A new family of modified Ostrowski's methods with accelerated eight order convergence, Numerical Algorithms, Vol. 54, no. 4, pp. 445-458 (2010).
\bibitem{Petkovic}
M. S. Petkovic: On a general clam of multipoint root-finding methods of high computational efficiency. SIAM J. Numerical Analysis, 47, pp. 4402-4414 (2010).
\bibitem{Sargolzaei}
P. Sargolzaei and F. Soleymani: Accurate fourteenth order methods for solving non-linear equations, Numerical Algorithm, vol. 58, pp. 513-527 (2011).
\bibitem{Thukral}
R. Thukral: A new eight-order iterative method for solving non-linear equations Applied Mathematics and Computation, vol. 217, no. 1, pp. 222-229 (2010).
\bibitem{Richard}
R. F. Kung: A family of fourth order methods for nonlinear equations, SIAM Journal on Numerical Analysis 10, pp.       876-879 (1973).
\bibitem{Khattvi}
S. K. Khattri and I. K. Argyrous, Sixth-order derivative free family of iterative methods, Applied mathematics and Computation, vol. 217, no. 12, pp. 5500-5507 (2011).
\bibitem{Weerakoon}
S. Weerakoon and T. G. I. Fernando: A variant of Newtons's Method with accelerated third-order convergence, method with accelerated third-order convergence, Applied Mathematics letters, vol. 13, no.8, pp. 87-93 (2000).
\bibitem{Gautschi}
W. Gautschi: Numerical Analysis: An Introduction Birkhauser, Barton, Mass, USA (1997).
\bibitem{Bi1}
W. Bi, H. Ren and Q. Wen: A new family of eight-order iterative methods for solving non linear equation, Applied Mathematics and Computation, vol. 214, no. 1, pp. 234-245 (2009).
\bibitem{Wang}
X. Wang and L. Liu: Modified Ostrowski's method with eight-order convergence and high efficiency index, Applied Mathematics Letters, 23, pp. 549-554 (2010).
\bibitem{Kim} 
Y. I. Kim: A try parametric family of three step eight order methods for solving non linear equations, International Journal of Computer Mathematics, vol. 89, no. 8, pp. 1051-1059 (2012).
\\\\
\end{thebibliography}
\end{document}